\documentclass[10pt]{article}
\usepackage{mathrsfs}
\usepackage{amsmath,amscd,amsthm,amsfonts}
\oddsidemargin=0.5cm\topmargin=-1cm \numberwithin{equation}{section}

\theoremstyle{plain}
\newtheorem{theorem}{theorem}[section]

\newtheorem{definition}[theorem]{Definition}
\newtheorem{lemma}[theorem]{Lemma}
\newtheorem{sublemma}[theorem]{Sublemma}

\newtheorem{remark}[theorem]{Remark}

\newtheorem*{ThmA}{Theorem A}
\newtheorem*{ThmB}{Theorem B}
\newtheorem*{ThmC}{Theorem C}
\newtheorem*{ThmD}{Theorem D}
\newtheorem*{ThmE}{Theorem E}
\newtheorem*{ThmB$'$}{Theorem B$'$}
\newtheorem*{Example}{Example}

\def\a{\alpha}
\def\b{\beta}

\def\ve{\varepsilon} 
   \def\S{\Sigma}
\def\t{\theta}

\def\Diff{\mathop{\hbox{{\rm Diff}}}}
\def\id{{\rm id}}

\def\Vol{\mathop{\hbox{\rm Vol}}}

\def\disp{\displaystyle}

\begin{document}
\vskip 5.1cm
\title{Quasi-Stability of Partially Hyperbolic Diffeomorphisms\footnotetext{
The second author is supported by NSFC(No:11071054), the Key Project of Chinese Ministry of Education(No:211020), NCET and the SRF for ROCS, SEM.}}
\author {Huyi Hu$^1$ and Yujun Zhu$^2$   \\
\small {1. Department of Mathematics,}\\
\small {Michigan State University, East Lansing, MI 48824, USA}\\
\small {2. College of Mathematics and Information Science,}\\
\small {Hebei Normal University, Shijiazhuang, 050024, P.R.China}}
\date{}
\maketitle

\begin{abstract}
A partially hyperbolic diffeomorphism $f$ is structurally
quasi-stable if for any diffeomorphism $g$ $C^1$-close to $f$,
there is a homeomorphism $\pi$ of $M$ such that $\pi\circ g$
and $f\circ\pi$ differ only by a motion $\tau$ along center directions.
$f$ is topologically quasi-stable if for any homeomorphism $g$ $C^0$-close to $f$,
the above holds for a continuous map $\pi$ instead of a homeomorphism.
We show that any  partially hyperbolic diffeomorphism $f$ is topologically
quasi-stable, and if $f$ has $C^1$ center foliation $W^c_f$, then $f$ is structurally quasi-stable.
As applications we obtain continuity of topological
entropy for certain partially hyperbolic diffeomorphisms with one
or two dimensional center foliation.
\end{abstract}

\section{Introduction}

The motivation of this paper is to study topological properties of
partially hyperbolic systems
which are similar to those of uniformly hyperbolic systems.

Partial hyperbolicity theory was first studied in the work of Brin and
Pesin (\cite{Brin}) which emerged in attempts to extend the notion
of complete hyperbolicity. A closely related notion of normal
hyperbolicity was introduced earlier by Hirsh, Pugh and Shub
\cite{Hirsch}. For general theory of partial hyperbolicity and
normal hyperbolicity, we refer to \cite{Pesin}, \cite{Hirsch1},
\cite{Barreira} and \cite{Bonatti}.

It is well known that Anosov diffeomorphisms are \emph{structurally
stable}(\cite{Anosov1}), that is, if $f$ is an Anosov diffeomorphism
on a compact manifold $M$ then any diffeomorphism $g$ $C^1$-close to
$f$ is topologically conjugate to $f$, i.e., there exists a
homeomorphism $\pi$ on $M$ such that
\begin{equation}\label{conjugacy}
\pi\circ g=f\circ\pi.
\end{equation}
Moreover, $f$ is also \emph{topologically stable}(\cite{Walters}),
that is, for any homeomorphism $g$ $C^0$-close to $f$, there exists
a continuous map $\pi$ from $M$ onto $M$ such that equation
(\ref{conjugacy}) holds.  For partially hyperbolic diffeomorphisms,
we can not expect such stabilities because of the existence
of the center direction.
However, since the systems have both stable and unstable directions,
we should be able to obtain some similar properties if we look at
the behavior of the hyperbolic part, and ``ignore'' the motions
along the center direction.

In this paper, we shall investigate the ``stability" property of
partially hyperbolic systems under $C^0$ and $C^1$ perturbations. Let
$f$ be a partially hyperbolic diffeomorphism. We show in Theorem A
that for any homeomorphism $g$ $C^0$-close to $f$, there exist a
continuous map $\pi$ from $M$ to itself and a family of locally
defined continuous maps $\{\tau_x:x\in M\}$, which send points
along the center direction, such that
\begin{equation}\label{qusi conjugacy}
\pi\circ g(x)=\tau_{f(x)}\circ f\circ\pi(x) \quad  \text{for all} \
x\in M.
\end{equation}
In this case we say that $f$ is \emph{topologically quasi-stable}.
Moreover, if center foliation $\mathcal{W}^c_f$ of $f$ exists and
is of $C^1$, then we can choose a new family $\{\tau_x:x\in M\}$,
which map points along the center leaves such that an equation similar
to (\ref{qusi conjugacy}) holds.
The results are given in Theorem~B.
Theorem B$'$ deal with a particular case,
i.e., one dimensional center foliation, in which the map
$\tau$ can be determined by a flow along the foliation. In Theorem C
we obtain \emph{structural quasi-stability} property of $f$
under $C^1$ perturbation. That is, if the center foliation $\mathcal{W}^c_f$
is $C^1$, then
for any diffeomorphism $g$ $C^1$-close to $f$, $\pi$ is a
leaf conjugacy between $f$ and $g$.

As applications of the results, we obtain that if $f$ is the time $1$
map of an Anosov flow generated by a $C^1$ vector field,
then any diffeomorphism $g$ $C^1$-close to $f$ is a time $1+\tau\circ f$
map of a flow, and the topological entropy of $f$ and $g$
are close (Theorem D).  Also, if $f$ has almost parallel center foliation
(see next section for the precise definition),
then so does any diffeomorphism $g$ $C^1$-close to $f$,
and the topological entropy function
is locally constant at $f$ in $\Diff^1(M)$ in the case of one dimensional
center foliation, and continuous at $f$ in $\Diff^\infty(M)$ in the case
of two dimensional center foliation (Theorem~E).

Our results concerning topological quasi-stability and structural
quasi-stability can be regarded as generalizations of topological
stability and structural stability for hyperbolic systems
(\cite{Anosov1} and \cite{Walters}) to partially hyperbolic systems.
They can also be regarded as generalizations of leaf conjugacy for
the case that $f$ has $C^1$ center foliation (\cite{Hirsch1} and
\cite{Pesin}). However, for topologically quasi-stability, we do not
require any additional assumption on $f$ except for partial
hyperbolicity. The methods we use for topological and structural
quasi-stability are basically the same.  We construct an operator in
the Banach space consisting all continuous sections of the tangent
bundle $TM$ which is contracting in a neighborhood of the zero
section such that under the inverse of the exponential map
$\exp^{-1}$,  $\pi$ and $\tau$ are given by the fixed point of the operator.
The methods are adopted from \cite{Walters, Kato} and are different
from that used in \cite{Hirsch1}.  We notice that there is another strategy to investigate the topological quasi-stability of partially hyperbolic diffeomorphism using the similar idea in \cite{Walters1}, in which Walters obtained topological stability for expansive homeomorphisms with  shadowing property. Actually, adapting the unified approach in this paper we will show in a forthcoming paper \cite{HZZ} that any partially hyperbolic diffeomorphism has the so-called ``quasi-shadowing" property. Hence, we can obtain the similar results in Theorem A, Theorem B and Theorem B$'$.

With some additional condition Y. Hua showed that the topological
entropy is continuous near the time one map $f$ of an Anosov flow
(\cite{Hua}). Now the fact becomes a direct consequence of our
result and a result in \cite{Thomas}.
Our results for topological entropy of diffeomorphisms
with almost parallel center foliation are similar to that in \cite{Xia},
which is under the assumption that the strongly unstable and
stable foliations stably carry some unique nontrivial homologies.

This paper is organized as the following.
The statements of results are given in Section \ref{DSN}. We also
define some words and notations in the section. In Section
\ref{STopqs} we deal with topological quasi-stability, including the
proof of Theorem~A, Theorem~B and Theorem B$'$. The case of
structural quasi-stability is discussed in Section~\ref{SStructqs},
where we use the facts obtained in Section~\ref{STopqs} to
prove that the map $\pi$ is a homeomorphism to obtain Theorem C.
Section~\ref{SApp} is concerning the applications to topological
entropy, where Theorem~D and Theorem~E are proved.

\section{Definition, statement of results and notations}\label{DSN}

Let $M$ be an $m$-dimensional $C^\infty$ compact Riemannian manifold.
We denote by $\|\cdot\|$ and $d(\cdot,\cdot)$ the norm on $TM$ and
the metric on $M$ induced by the Riemannian metric respectively.
Denote by $\Diff^r(M)$ the set of $C^r$ diffeomorphisms of $M$,
$1\le r\le \infty$.

A diffeomorphism $f\in\Diff^r(M)$ is said to be
\emph{(uniformly) partially hyperbolic} if there exist
numbers $\lambda,\lambda',\mu$ and $\mu'$ with
$0<\lambda<1<\mu$ and $\lambda<\lambda'\leq \mu'<\mu$,
and an invariant decomposition $T_xM=E_x^s\oplus E_x^c\oplus E_x^u$
\ $\forall x\in M$, such that for any $n\ge 0$,
\begin{eqnarray*}
\|d_xf^nv\|  &\!\!\!\!\!  \le C\lambda^n\|v\|\ \   &\text{as} \ v\in E^s(x), \\
C^{-1}(\lambda')^n\|v\| \le \|d_xf^nv\|  &\!\!\!  \le C(\mu')^n\|v\|
               \ &\text{as} \   v\in E^c(x),\\
C^{-1}\mu^n\,\|v\| \leq  \|d_xf^nv\|&\ &\text{as} \  v\in E^u(x)
\end{eqnarray*}
hold for some number $C>0$.  $E_x^s, E_x^c$ and $E_x^u$ are called
\emph{stable, center} and \emph{unstable} subspace, respectively.
Via a change of Riemannian metric we always assume that $C=1$.
Moreover, for simplicity of notation, we assume that
$\disp \lambda=\frac{1}{\mu}$.

Since $M$ is compact, we can take constant $\rho_0>0$ such that for
any $x\in M$, the standard exponential mapping
$\exp_{x}:\{v\in T_{x}M:\|v\|<\rho_0\}\to M$ is a $C^{\infty}$
diffeomorphism to the image. Clearly, we have $d(x, \exp_xv)=\|v\|$
for $v\in T_xM$ with $\|v\|<\rho_0$.
For any diffeomorphism $f: M\to M$, we take $\rho=\rho_f\in (0,\rho_0/2)$
such that for any $x,y\in M$ with $d(f^{-1}(x),y)\le \rho$, $v\in T_yM$
with $\|v\|\le \rho$,
$$d(x, f\circ \exp_yv)\le \rho_0/2.
$$
Reduce $\rho$ if necessary such that both sides in equation \eqref{feqn3ThmA}
and \eqref{feqn3ThmB}, in the proof of Theorem A and Theorem B
respectively, are contained in the set $\{v\in T_{x}M:\|v\|<\rho_0\}$.

For any given continuous center section $u=\{u(x)\in E_x^c:x\in M\}$ with $\sup_{x\in M}\|u(x)\|<\rho$,
we define a smooth map $\tau^{(1)}_x=\tau^{(1)}_x(\cdot, u)$
on $B(x,\rho)$ for any $x\in M$ by
$$
\tau^{(1)}_x(y)=\exp_x(u(x)+\exp_x^{-1}y).
$$

\begin{ThmA}\label{ThmA}
Let $f\in \Diff^r(M)$ be a partially hyperbolic diffeomorphism.
Then there exists $\varepsilon_0\in
(0,\rho)$ satisfying the following conditions: For any
$\varepsilon\in (0,\varepsilon_0)$ there exists $\delta>0$ such that
for any homeomorphism $g$ of $M$ with $d(f,g)<\delta$ there exist a
continuous center section $u$ and a surjective
continuous map $\pi:M\to M$ such that
\begin{equation}\label{feqnThmA}
\pi\circ g(x)=\tau^{(1)}_{f(x)}\circ f\circ\pi(x),\;\;x\in M.
\end{equation}

Moreover, $u$ and $\pi$ can be chosen uniquely so as to satisfy the
following conditions:
\begin{equation}\label{cond1}
\begin{split}
&d(\pi,{\id}_M)<\varepsilon,  \\
&\exp_x^{-1}(\pi(x))\in E_x^s\oplus E_x^u \ \  \text{for} \ x\in M.
\end{split}
\end{equation}
\end{ThmA}

We mention again here that the theorem does not require any additional
condition, provided $f$ is a partially hyperbolic diffeomorphism,
and $g$ is a homeomorphisms $C^0$ close to $f$.

If $f$ has $C^1$ center foliation
$\mathcal{W}^c_f$, then we can require $\tau$ to move along the center
foliation. In this case, we denote for any $\varepsilon>0$,
$\Sigma_\varepsilon(x)=\exp_x(H_x(\varepsilon))$, where $H_x(\varepsilon)$ is
the $\varepsilon-$ball in $E_x^s\oplus E_x^u$. Obviously, $\Sigma_\varepsilon(x)$ is a smooth
disk transversal to $E_x^c$ at $x$. Since the center foliation
$\mathcal{W}^c_f$ is $C^1$, we can conclude that if $y$ is close enough to $x$, then there exist a locally defined smooth map $\tau_x^{(2)}$ on
some neighborhood $U(x)$ of $x$ and a  constant $K_1>1$ independent of $x$ such that for any $y\in U(x)$, we have
\begin{equation}\label{cond1 of tau2}
\tau_x^{(2)}(y) \in \Sigma_\varepsilon(x)\cap \mathcal{W}^c_f(y)
\end{equation}
and
\begin{equation}\label{cond2 of tau2}
d(\tau_x^{(2)}(y),x)<K_1d(y,x).
\end{equation}

\begin{ThmB}\label{ThmB}
Assume that $f\in \Diff^r(M)$ is a partially hyperbolic diffeomorphism
with $C^1$ center
foliation $\mathcal{W}^c_f$. Then there exists $\varepsilon_0\in
(0,\rho)$ satisfying the following conditions: For any
$\varepsilon\in (0,\varepsilon_0)$ there exists $\delta>0$ such that
for any homeomorphism $g$ of $M$ with $d(f,g)<\delta$ there exists a surjective
continuous map $\pi:M\to M$ such that
\begin{equation}\label{feqnThmB}
\pi\circ g(x)=\tau^{(2)}_{g(x)}\circ f\circ\pi(x),\;\;x\in M.
\end{equation}

Moreover, $\pi$ can be chosen uniquely so as to satisfy the conditions in (\ref{cond1}).
\end{ThmB}

As a special case, if the center foliation $\mathcal{W}^c_f$ is $C^1$ and of
dimension one, then we can define $\tau$ more directly.
Let $u\in \mathfrak{X}^c$ with $\|u(x)\|=1$ for any $x\in M$,
and $\varphi^t$ be the flow generated by $u$. For any continuous
function $\tilde\tau: M\to \mathbb{R}$, define a smooth map
$\tau^{(3)}_x=\tau^{(3)}_x(\cdot, \tilde\tau)$ of $B(x,\rho)$ for any $x\in M$ by
$$
\tau^{(3)}_x(y)=\varphi^{\tilde\tau(x)}(y).
$$

\begin{ThmB$'$}\label{ThmB'}
Assume that $f\in \Diff^r(M)$ is a partially hyperbolic diffeomorphism
with one dimensional $C^1$ center foliation $\mathcal{W}^c_f$.
Then there exists $\varepsilon_0\in (0,\rho)$ satisfying the following
conditions: For any
$\varepsilon\in (0,\varepsilon_0)$ there exists $\delta>0$ such that
for any homeomorphism $g$ of $M$ with $d(f,g)<\delta$ there exists a
continuous function $\widetilde{\tau}:M\to \mathbb{R}$ and a
surjective continuous map $\pi:M\to M$ such that
\begin{equation}\label{feqnThmB'}
\pi\circ g(x)=\tau^{(3)}_{f(x)} \circ f\circ\pi(x), \quad x\in M.
\end{equation}

Moreover, $\tilde\tau$ and $\pi$ can be chosen uniquely so as to satisfy the conditions in (\ref{cond1}).
\end{ThmB$'$}

Now we consider the structural quasi-stability as $g$ is $C^1$-close to $f$.

A diffeomorphism $f$ is called \emph{dynamically coherent} if
$E^{cu}:=E^c\oplus E^u$, $E^{c}$, and $E^{cs}:=E^c\oplus E^s$ are
integrable, and everywhere tangent to $\mathcal{W}^{cu}_f$,
$\mathcal{W}^{c}_f$ and $\mathcal{W}^{cs}_f$, the {\it
center-unstable}, \emph{center} and \emph{center-stable foliations},
respectively; and $\mathcal{W}^c_f$ and $\mathcal{W}^u_f$ are
subfoliations of $\mathcal{W}^{cu}_f$, while $\mathcal{W}^c_f$ and
$\mathcal{W}^s_f$ are subfoliations of $\mathcal{W}^{cs}_f$. By
Theorem 2.3 of \cite{Pugh}, if $f$ is as in Theorem B then it is
dynamically coherent and this property is permanent under $C^1$
perturbation.

\begin{ThmC}\label{ThmC}
Under the assumption of Theorem B (resp. Theorem B$'$), if $g$ is a
diffeomorphism $C^1$-close to $f$, then $\pi$ can be chosen
to be a homeomorphism and hence there exists a homeomorphism $\tau^{(2)}$
in Theorem B (resp. $\tau^{(3)}$ in Theorem B$'$) such that
$\pi\circ g=\tau^{(2)}\circ f\circ\pi$
(resp. $\pi\circ g=\tau^{(3)}\circ f\circ\pi$).
Also, $\pi$ and $\tau^{(2)}$ (resp. $\tau^{(3)}$) can be chosen uniquely so as to satisfy the conditions in (\ref{cond1}) if we replace $E^s_x$ and $E^u_x$ in \eqref{cond1}
by their smooth approximation $\tilde E^s_x$ and $\tilde E^u_x$ respectively.

Moreover, $\pi$ sends
$\mathcal{W}^{cu}_g$, $\mathcal{W}^{c}_g$ and $\mathcal{W}^{cs}_g$ to
$\mathcal{W}^{cu}_f$, $\mathcal{W}^{c}_f$ and $\mathcal{W}^{cs}_f$ respectively.
In particular,  $\pi$ is a leaf conjugacy from $(g,\mathcal{W}_g^c)$ to
$(f, \mathcal{W}_f^c)$.
\end{ThmC}

\begin{remark}
In fact, if $\pi$ is one to one, then for any $y\in M$,
$x=\pi^{-1}(f^{-1}(y))$ is uniquely determined.
Hence we can define $\tau^{(2)}: M\to M$ by
$\tau^{(2)}(y)=\tau^{(2)}_{g(\pi^{-1}\circ f^{-1}(y))}(y)$ and obtain
$$
\pi\circ g(x)=\tau^{(2)}\circ f\circ\pi(x), \quad  x\in M.
$$

$\tau^{(3)}$ can be defined in a similar way.
\end{remark}

The main result of Theorem C, in particular, leaf conjugacy,
is well known (\cite{Hirsch1}, \cite{Pesin}).
Our proof provides a different approach.

\begin{Example}
Let $N$ be a smooth compact Riemannian manifold and
$h: N\longrightarrow N$ be an Anosov diffeomorphism. Then the
diffeomorphism
$$
f=h\times \id_{S^1}: N\times S^1\longrightarrow N\times S^1
$$
is quasi-stable. In particular, if $R$ is a rotation on $S^1$
close to the identity, and
$$
g=h\times R: N\times S^1\longrightarrow N\times S^1,
$$
then we can take $\pi=\id_{N\times S^1}$ and $\tau^{(2)}=\id_N\times R$
in Theorem C.

Moreover, Theorem E below gives that the topological entropy
is constant in a neighborhood of $f$ in $C^1$ topology.
\end{Example}

As applications of Theorem C, we have the following results about
the continuity of the entropy.

We say that a diffeomorphism $g$ is a time $1+\tau$ map of a flow $\psi$
for some real function $\tau$ on $M$ if $g(x)=\psi^{1+\tau(x)}(x)$
for any $x\in M$.

\begin{ThmD}\label{ThmD}
Let $f$ be the time one map of an Anosov flow $\varphi$.
Then for any diffeomorphism $g$ $C^1$-close to $f$,
there is a flow $\psi$ and a continuous real function $\tau$
on $M$ such that $g$ is the time $1+\tau\circ f$ map of the flow $\psi$.
Hence the topological entropy function is continuous at $f$ in $\Diff^1(M)$
with $C^1$ topology.
\end{ThmD}

Y. Hua proved the second part of the theorem under the condition that
$f$ is topologically transitive (\cite{Hua}).

Let $f$ be a partially hyperbolic diffeomorphism with center
foliation $\mathcal{W}^c_f$. For a smooth surface $\S$, we denote
$\S\perp \mathcal{W}^c_f$ if for any $x\in \S\cap \mathcal{W}^c_f$,
$T_x\S\perp E^c_f(x)$. For any two surfaces $\S_1$ and
$\S_2$ that are smooth or are images of homeomorphisms of some
smooth surfaces, the holonomy map $\t^c: \S_1\to \S_2$ is a
continuous map defined by sliding along the
$\mathcal{W}^c_f$-leaves, i.e. for $x\in\Sigma_1$,
$\theta^c(x)\in\Sigma_2\cap \mathcal{W}^c_f(x)$.

\begin{definition}
A partially hyperbolic diffeomorphism $f$ with integrable center
foliation $\mathcal{W}^c_f$ is said to have  \emph{almost parallel
center foliation} if for any $\alpha>0$, there exists constant
$\b>0$, such that for any smooth surfaces $\S_1, \S_2$ with
$\S_1\perp\mathcal{W}^c_f$ and $\S_2\perp \mathcal{W}^c_f$, for any
$x, y\in \S_1$ with $d(x,y)\le \beta$, we have $d(\t^c(x),
\t^c(y))\le \alpha$ whenever they are defined.
\end{definition}

\begin{remark}
In the definition we require $\S_1\perp \mathcal{W}^c_f$ and
$\S_2\perp \mathcal{W}^c_f$ only for convenience. It is clear that
we can change the definition by requiring the angles between $\S_1$,
$\S_2$ and $\mathcal{W}^c_f$ uniformly bounded from below.
\end{remark}

\begin{remark}
The requirements of the definition mean that
the holonomy maps along the center foliation are equicontinuous
for all possible holonomy maps whenever they are defined.
\end{remark}

It is clear that each of the maps $f$ and $g$ given by the above example has
almost parallel center foliation, while Anosov flows do not.

\begin{ThmE}\label{ThmE}
Let $f$ be a partially hyperbolic diffeomorphism as in Theorem B.
If the center foliation of $f$ is  almost parallel, then any
diffeomorphism $g$ that is $C^1$-close to $f$ also has almost
parallel center foliation.

Moreover, if the center foliation of $f$ is one dimensional,
then the topological entropy function is locally constant
in $\Diff^1(M)$;
if $f\in\Diff^\infty(M)$ and the center foliation of $f$ is two dimensional,
then the topological entropy function is continuous near $f$
in $\Diff^\infty(M)$ with $C^1$ topology.

\end{ThmE}

For the case that the dimension of center subbundle of $f$ is
one or two, the same conclusions are obtained in \cite{Xia}
under the assumption that the strongly stable and unstable foliations of $f$
stably carry some unique nontrivial homologies.
Our proof uses some idea in the paper.

Denote by $\mathfrak{X}$ the Banach space of continuous vector fields
on $M$ with the norm
$$
\|\omega\|=\sup_{x\in M}\|\omega(x)\|,  \quad w \in \mathfrak{X}.
$$
In other words, each element of $\mathfrak{X}$ is a continuous
section of the tangent bundle $TM$.
Similarly, we denote  by $\mathfrak{X}^s,\mathfrak{X}^c$ and
$\mathfrak{X}^u$ the space of continuous sections of the stable,
center and unstable bundles $E^s, E^c$ and $E^u$ respectively.
Also, we denote
$\mathfrak{X}^{us}=\mathfrak{X}^u\oplus \mathfrak{X}^s$. Let
$\Pi^{s}_x:T_{x}M\to E^{s}_x$ be the projection onto
$E^{s}_x$ along $E^c_x\oplus E^u_x$. $\Pi^{c}_x$ and $\Pi^{u}_x$ are
defined in a similar way.

Recall that $\|\cdot\|$ is the norm on $TM$. We define the norm
$\|\cdot\|_1$ on $TM$ by $\|w\|_1=\|u\|+\|v\|$ if $w=u+v\in T_xM$
with $u\in E^c_x$ and $v\in E^u_x\oplus E^s_x$. Similarly,
if $w=u+v\in \mathfrak{X}$ with $u\in \mathfrak{X}^c$ and
$v\in\mathfrak{X}^{us}$, we also define $\|w\|_1=\|u\|+\|v\|$. By
triangle inequality and the fact that the angles between $E^c$ and
$E^u\oplus E^s$ are uniformly bounded away from zero, we know that there
exists a constant $L$ such that
\begin{equation}\label{fdefL}
\|w\|\le \|w\|_1\le L\|w\|.
\end{equation}

For any $\ve>0$, we denote
$$
\mathfrak{B}(\ve)=\{w\in \mathfrak{X}: \|w\|\le \ve\},\qquad
\mathfrak{B}^{us}(\ve)=\{w\in \mathfrak{X}^{us}: \|w\|\le \ve\},
$$
$$
\mathfrak{B}_1(\ve)=\{w\in \mathfrak{X}: \|w\|_1\le \ve\}.
$$

\section{Topological quasi-stability}\label{STopqs}

\subsection{The general case}

\begin{proof}[Proof of Theorem A]
We choose
\begin{equation}\label{epsilon0}
\varepsilon_0\in (0,\rho)
\end{equation}
small enough such that any map $\pi$ with $d(\pi,\id_M)<\varepsilon_0$ must be
surjective (see e.g. Lemma 3 of \cite{Walters} for existence of such $\ve_0$).

To find a continuous center section $u\in
\mathfrak{X}^c$ and a surjective continuous map
$\pi:M\longrightarrow M$ satisfying \eqref{feqnThmA} and the conditions in
(\ref{cond1}) of this theorem, we shall first try to solve the
equation
\begin{equation}\label{feqn2ThmA}
\pi\circ g(x)=\tau^{(1)}_{f(x)}\circ f\circ\pi(x)
\end{equation}
for unknown $u$ and $\pi$. Putting $h=g\circ f^{-1}$ and
$\pi(x)=\exp_x(v(x))$ for $v\in \mathfrak{B}^{us}(\rho)$, and
replacing $x$ by $f^{-1}(x)$, we see that \eqref{feqnThmA} is
equivalent to
\begin{equation}\label{feqn3ThmA}
\exp_x^{-1}\circ \exp_{h(x)}\bigl(v(h(x))\bigr)
=\exp_x^{-1}\circ\tau^{(1)}_x\circ f
   \circ \exp_{f^{-1}(x)}\bigl(v(f^{-1}(x))\bigr).
\end{equation}

By the definition of $\tau^{(1)}_x$, we have
\begin{equation*}\label{equation5 of lemma01}
\exp_x^{-1}\circ\tau^{(1)}_x\circ f
\circ\exp_{f^{-1}(x)}\bigl(v(f^{-1}(x))\bigr) =u(x)+
\exp_x^{-1}\circ f\circ\exp_{f^{-1}(x)}\bigl((v(f^{-1}(x))\bigr).
\end{equation*}
Define an operator $\beta: \mathfrak{B}(\rho)\to \mathfrak{X}$ and a
linear operator $F: \mathfrak{X}\to \mathfrak{X}$ by
\begin{eqnarray}
\label{fdefbeta} &&\beta(w)(x)=\exp_x^{-1}\circ f\circ
        \exp_{f^{-1}(x)}\bigl((w(f^{-1}(x))\bigr) , \\
\label{fdefF} &&(Fw)(x)=d_{f^{-1}(x)}fw(f^{-1}(x)).
\end{eqnarray}
Clearly, $Fw=d_0\beta w$. Let
\begin{equation}\label{fdefs}
\eta(w)(x)=\beta(w)\bigl(x\bigr) - (d_0\beta w)\bigl(x\bigr).
\end{equation}
Then we can write
\begin{equation}\label{feqn3right}
\exp_x^{-1}\circ\tau^{(1)}_x\circ f\circ
\exp_{f^{-1}(x)}(v(f^{-1}(x))) =(Fv)(x)+u(x)+\eta(v)(x).
\end{equation}

Define a linear operator $J_h: \mathfrak{B}(\rho)\to\mathfrak{X}$ by
\begin{equation}\label{fdefJh}
(J_hw)(x)=\sum_{i=s,c,u}\Pi_x^i\circ
d_0(\exp_x^{-1}\circ\exp_{h(x)})\circ\Pi_{h(x)}^iw(h(x))
\end{equation}
for any $w\in \mathfrak{B}(\rho)$. Set
\begin{equation}\label{}
\begin{aligned}
\theta_h(w)(x)
=&\exp_x^{-1}\circ\exp_{h(x)}(w(h(x)))-d_0(\exp_x^{-1}
  \circ\exp_{h(x)})w(h(x))\notag\\
+&\sum_{i,j=s,c,u,i\neq j}\Pi_x^i\circ
d_0(\exp_x^{-1}\circ\exp_{h(x)})\circ\Pi_{h(x)}^jw(h(x)).\notag
\end{aligned}
\end{equation}
Then we have
\begin{equation}\label{feqn3left}
\exp_x^{-1}\circ \exp_{h(x)}(v(h(x)))=(J_hv)(x)+\theta_h(v)(x).
\end{equation}
Also we mention that by definition,
\begin{equation}\label{ftheta0h}
\theta_h(0)(x)=\exp_x^{-1}\circ\exp_{h(x)} 0 =\exp_x^{-1} h(x).
\end{equation}

Therefore, by \eqref{feqn3left} and \eqref{feqn3right},
\eqref{feqn3ThmA} is equivalent to
$$
J_hv+\theta_h(v)=Fv+u+\eta(v),
$$
further, is equivalent to
$$
-J_h^{-1}u+(\id_{\mathfrak{X}}-J_h^{-1}F)v=J_h^{-1}(\eta(v)-\theta_h(v)).
$$
Define a linear operator $P_h$ from a neighborhood of $0\in
\mathfrak{X}$ to $\mathfrak{X}$ by
\begin{equation}\label{Ph0}
P_h\omega=-J_h^{-1}u+(\id_{\mathfrak{X}}-J_h^{-1}F)v
\end{equation}
for $\omega=u+v\in \mathfrak{X}$, where $u\in \mathfrak{X}^c$ and
$v\in \mathfrak{X}^{us}$.

Define an operator $\Phi_h$ from a neighborhood of $0\in
\mathfrak{X}$ to $\mathfrak{X}$ by
$$
\Phi_h(u+v)=P_h^{-1}J_h^{-1}(\eta(v)-\theta_h(v)).
$$
Hence, equation (\ref{feqn2ThmA}) is equivalent to
\begin{equation}\label{feqn4ThmA}
\Phi_h(u+v)=u+v,
\end{equation}
namely, $u+v$ is a fixed point of $\Phi_h$.

By Lemma~\ref{LThmA} below, we know that for any $\ve\in
(0,\varepsilon_0)$ there exists $\delta=\delta(\ve)$ such that for any homeomorphism
$h$ with $d(h,\id_M)\le \delta$, $\Phi_h: \mathfrak{B}_1(\ve)\to
\mathfrak{B}_1(\ve)$ is a contracting map, and therefore has a fixed
point in $\mathfrak{B}_1(\ve)$. Hence, (\ref{feqn2ThmA}) has a
unique solution.
\end{proof}

\begin{lemma}\label{LThmA}
We can reduce $\varepsilon_0$ in (\ref{epsilon0}) if necessary such that for any $\ve\in
(0,\varepsilon_0)$ there exists $\delta=\delta(\ve)>0$ such
that for any homeomorphism $h$ of $M$ with $d(h,\id_M)\le \delta$,
$\Phi_h(\mathfrak{B}_1(\ve))\subset \mathfrak{B}_1(\ve)$ and for any
$\omega,\omega'\in \mathfrak{B}_1(\ve)$,
$$
\|\Phi_h(\omega)-\Phi_h(\omega')\|_1 \leq
\frac{1}{2}\|\omega-\omega'\|_1.
$$
\end{lemma}

\begin{proof}
Recall that the constant $L$ is given in \eqref{fdefL}.

Reduce $\varepsilon_0$ if necessary such that for any $\varepsilon\in (0,\varepsilon_0)$,
\begin{equation}\label{f1LThmA}
\frac{4L }{1-\lambda}C(\varepsilon)<\frac{1}{4},
\end{equation}
where $C(\varepsilon)$ is the Lipschitz constant of $\eta(v)$ given in
Sublemma \ref{SL2ThmA}.  This is possible since by the sublemma,
$C(\varepsilon)\to 0$ as $\ve\to 0$. Note that $\varepsilon$ only
depends on $f$.

Then we take $\delta=\delta(\varepsilon)$ such that
\begin{equation}\label{f2LThmA}
\frac{4L }{1-\lambda}\delta<\frac{1}{4}\varepsilon;
\end{equation}
and such that for any homeomorphism $h$ with $d(h,\id_M)<\delta$,
\begin{equation}\label{f3LThmA}
\max\{\|J_h\|,\;\|J_h^{-1}\|\} \le \min\left\{2, \
\frac{1+\lambda^{-1}}{2}\right\},
\end{equation}
where $J_h$ is defined in \eqref{fdefJh}, and
\begin{equation}\label{f4LThmA}
\frac{4L }{1-\lambda}K(h)<\frac{1}{4},
\end{equation}
where $K(h)$ is the Lipschitz constant of $\theta_h(\cdot)$ given in
Sublemma \ref{SL3ThmA}.  This is possible since by the sublemma,
$K(h)\to 0$ as $\ve\to 0$.

By \eqref{f3LThmA}, Sublemma~\ref{SL1ThmA} below can be applied and
therefore we get
\begin{equation}\label{f5LThmA}
\|P_h^{-1}\|_1\leq \frac{2}{1-\lambda}.
\end{equation}
Note that $J_h(\mathfrak{X}^i)=\mathfrak{X}^i$ for $i=s,c,u$. Then
it is easy to check that $\|J_h^{-1}\|_1\le \|J_h^{-1}\|$. Hence by
\eqref{f3LThmA}, we have  $\|J_h^{-1}\|_1\le 2$. Also, by
\eqref{fdefbeta}, $\beta(0)=0$ and therefore by \eqref{fdefs},
$\eta(0)=0$; and by \eqref{ftheta0h}, $\|\theta_h(0)\|\le \delta$.

Take  $\omega=u+v\in \mathfrak{B}_1(\ve)$ with $u\in \mathfrak{X}^c$
and $v\in \mathfrak{X}^{us}$. By using the above estimates,
Sublemma~\ref{SL2ThmA} and Sublemma~\ref{SL3ThmA}, and then \eqref{f1LThmA},
\eqref{f2LThmA} and \eqref{f4LThmA}, we can get
\begin{equation*}
\begin{split}
\|\Phi_h(\omega)\|_1
&\leq \|P_h^{-1}\|_1\cdot\|J_h^{-1}\|_1\cdot\|\eta(v)-\theta_h(v)\|_1 \\
&\leq \frac{2}{1-\lambda}\cdot 2\cdot L  \|\eta(v)-\theta_h(v)\| \\
&\leq \frac{4L  }{1-\lambda}
   (\|\eta(v)-\eta(0)\|+\|\theta_h(v)-\theta_h(0)\|+\|\theta_h(0)\|)\\
&\leq \frac{4L }{1-\lambda}
   (C(\varepsilon)\|\omega\|_1+K(h)\|\omega\|_1+\delta)\\
&\leq \frac{1}{4}\|\omega\|_1+\frac{1}{4}\|\omega\|_1
    +\frac{1}{4}\varepsilon
\leq \frac{3\varepsilon}{4},
\end{split}
\end{equation*}
which implies that
$\Phi_h(\mathfrak{B}_1(\ve))\subset\mathfrak{B}_1(\ve)$.

Similarly, for two elements $\omega=u+v,\;\omega'=u'+v'\in
\mathfrak{B}_1(\ve)$ with $u,u'\in \mathfrak{X}^c$ and $v,v'\in
\mathfrak{X}^{us}$, we have
\begin{equation*}
\begin{split}
\|\Phi_h(\omega)-\Phi_h(\omega')\|_1 &\leq\frac{4}{1-\lambda}
     (\|\eta(v)-\eta(v')\|_1+\|\theta_h(v)-\theta_h(v')\|_1)\\
&\leq\frac{4L }{1-\lambda}(\|\eta(v)-\eta(v')\|+\|\theta_h(v)-\theta_h(v')\|)\\
&\leq\frac{4L }{1-\lambda}(C(\varepsilon_0)\|\omega-\omega'\|_1
       +K(h)\|\omega-\omega'\|_1) \\
&\leq\frac{1}{2}\|\omega-\omega'\|_1.
\end{split}
\end{equation*}
This proves that $\Phi_h: \mathfrak{B}_1(\ve)\to
\mathfrak{B}_1(\ve)$ is a contraction.
\end{proof}

\begin{sublemma}\label{SL1ThmA}
For any homeomorphism $h$ of $M$ such that $J_h$ satisfies
\eqref{f3LThmA}, $P_h$ is invertible and
\begin{eqnarray}\label{f1SLThmA}
\|P_h^{-1}\|_1\leq \frac{2}{1-\lambda}.
\end{eqnarray}
\end{sublemma}

\begin{proof}
By the definitions of $F$ and $J_h$, we have
$F(\mathfrak{X}^i)=\mathfrak{X}^i$ and
$J_h(\mathfrak{X}^i)=\mathfrak{X}^i$ for $i=u,s,c$. Let
$F^{i}=F|_{\mathfrak{X}^i}$, $J_h^{i}=J_h|_{\mathfrak{X}^i}$ for
$i=u,s,c$. By the definition of $P_h$ we have
$P_h|_{\mathfrak{X}^i}=id_{\mathfrak{X}^i}-(J_h^{i})^{-1}\circ
F^{i},\;i=s,u$, and $P_h|_{\mathfrak{X}^c}=-(J_h^c)^{-1}$. So we
also have $P_h(\mathfrak{X}^i)=\mathfrak{X}^i$ for $i=s,c,u$.

Since $ \|F^s\|, \|(F^u)^{-1}\|\le \lambda$, by \eqref{f3LThmA} we
know that
$$\|(J_h^{s})^{-1}\circ F^{s}\|, \ \|((F^{(u)})^{-1}\circ J_h^{(u)})\|
\le \lambda\cdot (1+\lambda^{-1})/2 =(1+\lambda)/2 <1.
$$
Hence, both $P_h|_{\mathfrak{X}^s}$ and $P_h|_{\mathfrak{X}^u}$ are
invertible and
\begin{equation*}
\begin{split}
(P_h|_{\mathfrak{X}^s})^{-1}
&=(id_{\mathfrak{X}^s}-(J_h^{s})^{-1}\circ F^{s})^{-1}
=\sum_{k=0}^\infty((J_h^{s})^{-1}\circ F^{s})^k,      \\
(P_h|_{\mathfrak{X}^u})^{-1}
&=(id_{\mathfrak{X}^u}-(J_h^{u})^{-1}\circ
F^{u})^{-1}=-\sum_{k=1}^\infty((F^{u})^{-1}\circ J_h^{u})^k.
\end{split}
\end{equation*}
It follows that
\begin{eqnarray*}
\|(P_h|_{\mathfrak{X}^{us}})^{-1}\| \le
\max\left\{\|(P_h|_{\mathfrak{X}^s})^{-1}\|,
               \|(P_h|_{\mathfrak{X}^u})^{-1}\|\right\}
\le \frac{1}{1-(1+\lambda)/2} =   \frac{2}{1-\lambda}.
\end{eqnarray*}
By \eqref{f3LThmA} we also have
\begin{equation*}
\|(P_h|_{\mathfrak{X}^c})^{-1}\| \leq\|J_h\| \le 2.
\end{equation*}
So we obtain
$$
\|P_h^{-1}\|_1 \le \max\left\{\|(P_h|_{\mathfrak{X}^{us}})^{-1}\|,
               \|(P_h|_{\mathfrak{X}^c})^{-1}\|\right\}
\le \frac{2}{1-\lambda}.
$$
This is what we need.
\end{proof}

\begin{sublemma}\label{SL2ThmA}
For any $0<\varepsilon\le \rho$, there exists constant
$C(\varepsilon)>0$ such that for any $w, w'\in \mathfrak{B}(\ve)$,
\begin{equation*}\label{equation3 of lemma01}
\|\eta(w')-\eta(w)\|\leq C(\varepsilon)(\|w'-w\|).
\end{equation*}

Moreover, $C(\varepsilon)$ can be chosen in such a way that
$C(\varepsilon)\to 0$ as $\varepsilon\to 0$.
\end{sublemma}

\begin{proof}
Let $\beta_x:T_{f^{-1}(x)}M\longrightarrow T_x M$ be the map defined
by $\beta_x(\xi)=\exp_x^{-1}\circ f\circ
        \exp_{f^{-1}(x)}\xi$. Therefore $\beta(w)(f^{-1}(x))=\beta_x(w(f^{-1}(x)))$ for any $w\in \mathfrak{X}$.  Then by
\eqref{fdefs}, for any $w, w'\in \mathfrak{X}$ with
$\|w\|,\|w'\|<\varepsilon$,
\begin{equation*}
\begin{split}
&\|\eta(w')(x)-\eta(w)(x)\| \\
=&\|\beta_x(w'(f^{-1}(x)))-\beta_x(w(f^{-1}(x)))-d_0\beta_x(w'(f^{-1}(x))-w(f^{-1}(x)))\|  \\
=&\Bigl\|\int_0^1 d_{w(f^{-1}(x))+t(w'(f^{-1}(x))-w(f^{-1}(x)))}\beta_x(w'(f^{-1}(x))-w(f^{-1}(x)))dt\\
&-d_0\beta_x(w'(f^{-1}(x))-w(f^{-1}(x)))\Bigr\|\\
\leq& \sup_{w^*\in
I_x}\|d_{w^*}\beta_x-d_0\beta_x\|\cdot\|w'(f^{-1}(x))-w(f^{-1}(x))\|,
\end{split}
\end{equation*}
where $I_x=\{w(f^{-1}(x))+t(w'(f^{-1}(x))-w(f^{-1}(x))):t\in
[0,1]\}$. Since $d_{w^*} \beta_x$ is continuous with $w^*$ and the
continuity is uniform with respect to $x$, we can take
$$
C(\varepsilon) =\sup\bigl\{\|d_{w(f^{-1}(x))}\beta_x-d_0\beta_x\|: \
    w \in \mathfrak{B}_1(\varepsilon), x\in M\bigr\}.
$$
Now the results of the lemma are clear.
\end{proof}

\begin{sublemma}\label{SL3ThmA}
For any $h$ with $d(h,\id_M)\le \rho$, there exists a constant
$K=K(h)>0$ such that for any $w, w'\in \mathfrak{B}(\ve)$,
\begin{equation*}\label{equation3 of lemma02}
\|\theta_h(w')-\theta_h(w)\|\leq K(h)\|w'-w\|.
\end{equation*}

Moreover, $K(h)$ can be chosen in such a way that $K(h)\to 0$ as
$d(h,\id_M)\to 0$.
\end{sublemma}

\begin{proof}
Since the map $\theta$ is $C^\infty$ with respect to $w$ and $h$, we
can use the same method in the proof of the previous lemma to get
the inequality.

Note that if $h=\id_M$, then the partial derivative of $\theta$ with
respect to $w$ is zero.  So we get $K(h)\to 0$ as $d(h,\id_M)\to 0$.
\end{proof}

\subsection{The center foliation $\mathcal{W}^c_f$ is $C^1$}

\subsubsection{The general case}

\begin{proof}[Proof of Theorem B]

The proof is similar to that of Theorem A.

To find $\pi$ satisfying (\ref{feqnThmB}) and the
conditions in (\ref{cond1}) of this theorem, we shall try to solve the equation
\begin{equation}\label{feqn2ThmB}
\pi\circ g(x)=\tau^{(2)}_{g(x)}\circ f\circ\pi(x)
\end{equation}
for unknown $\pi$. Putting $\pi(x)=\exp_x(v(x))$ with $v\in \mathfrak{B}^{us}(\rho)$, we
see that (\ref{feqn2ThmB}) is equivalent to
\begin{equation}\label{feqn3ThmB}
v(x)=\exp_x^{-1}\circ\tau^{(2)}_x\circ f\circ
\pi(v({g^{-1}(x)})).
\end{equation}

Define an operator $\beta: \mathfrak{B}^{us}(\rho)\to \mathfrak{X}^{us}$ and a linear operator $A: \mathfrak{B}^{us}(\rho)\to \mathfrak{X}^{us}$ by
\begin{equation}\label{betaThmB}
(\beta(v))(x)=\exp_x^{-1}\circ\tau^{(2)}_x\circ f\circ
\pi(v({g^{-1}(x)})),
\end{equation}
\begin{equation}\label{AThmB}
(Av)(x)
=(A_{g^{-1}(x)}^s+A_{g^{-1}(x)}^u)v(g^{-1}(x)),
\end{equation}
where
$$
A_{g^{-1}(x)}^s=\Pi_x^s\circ
d_0(\exp_x^{-1}\circ\tau^{(2)}_x\circ f\circ
\exp_{g^{-1}(x)})\circ \Pi_{g^{-1}(x)}^s
$$
and
$$
A_{g^{-1}(x)}^u=\Pi_x^u\circ
d_0(\exp_x^{-1}\circ\tau^{(2)}_x\circ f\circ
\exp_{g^{-1}(x)})\circ \Pi_{g^{-1}(x)}^u.
$$
Let $\eta=\beta-A$.  By (\ref{betaThmB}) and (\ref{AThmB}),
(\ref{feqn3ThmB}) is equivalent to
$$
v=Av+\eta(v),
$$
further, is equivalent to
$$
v-Av=\eta(v).
$$
Define a linear operator $P$ from a neighborhood of $0\in
\mathfrak{X}^{us}$ to $\mathfrak{X}^{us}$ by
\begin{equation}\label{P}
Pv=(\id_{\mathfrak{X}^{us}}-A)v
\end{equation}
for $v\in \mathfrak{X}^{us}$.

Define an operator $\Phi$ from a neighborhood of $0\in
\mathfrak{X}^{us}$ to $\mathfrak{X}^{us}$ by
$$
\Phi(v)=P^{-1}\eta(v).
$$
Hence, the equation (\ref{feqn2ThmB}) is equivalent to
\begin{equation}\label{feqn4ThmB}
\Phi(v)=v,
\end{equation}
namely, $v$ is a fixed point of $\Phi$.

The remaining work is to show that for any $\ve\in
(0,\varepsilon_0)$ there exists $\delta=\delta(\ve)$ such that for any homeomorphism
$g$ with $d(g,f)\le \delta$, $\Phi: \mathfrak{B}(\ve)\to
\mathfrak{B}(\ve)$ is a contracting map, and therefore has a fixed
point in $\mathfrak{B}(\ve)$. Hence, (\ref{feqn2ThmB}) has a
unique solution. To this end we only need to modify the proof of Lemma~\ref{LThmA} to a easer version since in this case we need not to concern with the center direction. We leave the details to the reader.
\end{proof}

\subsubsection{$\mathcal{W}^c_f$ is of one dimensional}

\begin{proof}[Proof of Theorem B$'$]
The proof is also similar to that of Theorem A.

To find $\pi$ and $\tilde{\tau}$ satisfying (\ref{feqnThmB'}) and the
conditions in (\ref{cond1}) of this theorem, we shall
try to solve the equation
\begin{equation}\label{feqn2ThmB'}
\pi\circ g(x)=\tau^{(3)}_{f(x)}\circ f\circ\pi(x)
\end{equation}
for unknown $\tilde{\tau}$ and $\pi$. Putting $h=g\circ f^{-1}$ and
$\pi(x)=\exp_x(v(x))$ with $v\in \mathfrak{X}^{us}$, we see that
(\ref{feqn2ThmB'}) is equivalent to
\begin{equation}\label{feqn3ThmB'}
\exp_x^{-1}\circ
\exp_{h(x)}(v(h(x)))=\exp_x^{-1}\circ\varphi^{\tilde{\tau}(x)}\circ
f\circ \exp_{f^{-1}(x)}\bigl(v(f^{-1}(x))\bigr).
\end{equation}

Define $\beta: \mathfrak{B}(\rho)\times \mathfrak{C}(\rho)\to
\mathfrak{X}$, where $\mathfrak{C}(\rho)=\{\tilde{\tau}\in C^0(M):
\|\tilde{\tau}\|\le \rho\}$, by
\begin{eqnarray}\label{fdefbetaB'}
\beta(\omega,\tilde{\tau})(x)=\exp_x^{-1}\circ
\varphi^{\tilde{\tau}(x)} \circ f
  \circ \exp_{f^{-1}(x)}\bigl(\omega(f^{-1}(x))\bigr).
\end{eqnarray}
It is easy to see that
$$
\bigl(d_{(0,0)}\beta(\omega,\tilde{\tau})\bigr)(x)=(F\omega)(x)+\tilde{\tau}(x)\cdot
u(x),
$$
where $F: \mathfrak{X}\to \mathfrak{X}$ is defined in \eqref{fdefF}
(also recall that in this case $u$ is a unit center vector field).
Let
\begin{equation}\label{fdefsB'}
\eta(\omega,\tilde{\tau})(x)=\beta\bigl(\omega,\tilde{\tau}\bigr)(x)
    - \bigl(d_{(0,0)}\beta(\omega,\tilde{\tau})\bigr)(x).
\end{equation}
Then we can write
\begin{equation}\label{feqn3rightB'}
\exp_x^{-1}\circ\varphi^{\tilde{\tau}(x)}\circ f\circ
\exp_{f^{-1}(x)}(v(f^{-1}(x))) =(Fv)(x)+\tilde{\tau}(x)\cdot
u(x)+\eta(v,\tilde{\tau})(x).
\end{equation}

By \eqref{feqn3left} and \eqref{feqn3rightB'},  \eqref{feqn3ThmB'} is
equivalent to
$$
J_hv+\theta_h(v)=Fv+\tilde{\tau}\cdot u+\eta(v,\tilde{\tau}).
$$
where $J_h$ is a linear operator defined in \eqref{fdefJh}.
Further, the equation is equivalent to
$$
-J_h^{-1}(\tilde{\tau}\cdot
u)+(\id_{\mathfrak{X}}-J_h^{-1}F)v=J_h^{-1}(\eta(v,\tilde{\tau})-\theta_h(v)).
$$
Similarly we define a linear map $P_h$ by
\begin{equation}\label{fdefPhB'}
P_h\omega=-J_h^{-1}(\tilde{\tau}\cdot
u)+(\id_{\mathfrak{X}}-J_h^{-1}F)v
\end{equation}
for $\omega=\tilde{\tau}\cdot u+v\in \mathfrak{X}$ with $v\in
\mathfrak{B}^{us}(\rho)$ and $\tilde{\tau}\in \mathfrak{C}(\rho)$.
Hence, the above equation becomes
$$
P_h(\tilde{\tau}\cdot u+v)=J_h^{-1}(\eta(v,\tilde{\tau})-\theta_h(v)).
$$

Define a map $\Phi_h$ from a neighborhood of
$0\in \mathfrak{X}$ to $\mathfrak{X}$ by
$$
\Phi_h(\tilde{\tau}\cdot
u+v)=P_h^{-1}J_h^{-1}(\eta(v,\tilde{\tau})-\theta_h(v)).
$$
Hence, the equation (\ref{feqn2ThmB'}) is equivalent to
\begin{equation}\label{feqn4ThmB'}
\Phi_h(\tilde{\tau}\cdot u+v)=\tilde{\tau}\cdot u+v,
\end{equation}
namely, $\tilde{\tau}\cdot u+v$ is a fixed point of $\Phi_h$.

Also similar to what we have done in the proof of Theorem A, there exists
$\ve_0\in (0,\rho)$ such that for any $\ve\in
(0,\varepsilon_0)$ there exists $\delta=\delta(\ve)$ such that for any
homeomorphism $h$ with $d(h,\id_M)\le \delta$, $\Phi_h:
\mathfrak{B}_1(\ve)\to \mathfrak{B}_1(\ve)$ is a contracting map,
and therefore has a fixed point in $\mathfrak{B}_1(\ve)$. Hence,
(\ref{feqn2ThmB'}) has a unique solution. We leave the details to
the reader.
\end{proof}

\section{Structural quasi-stability}\label{SStructqs}

\begin{proof}[Proof of Theorem C]
We only prove this theorem under the assumption of Theorem B. We shall find $\pi$ and $u$ using the similar strategy in the proof of Theorem B. Furthermore, in order to
obtain a leaf conjugacy $\pi$ we shall give some necessary modification in techniques.

Since the center foliation of $f$ is $C^1$, hence from Theorem 5.10
of \cite{Pesin} and Section 6 of \cite{Hirsch}, we know that if
a diffeomorphism $g$ is sufficiently close to $f$ in $C^1$ topology,
then it is also partially hyperbolic, the corresponding splitting
$E_g^s\oplus E_g^c\oplus E_g^u$ is near that of $f$ and the center distribution
$E_g^c$ is integrable. Now choose $C^1$ bundle $\tilde{E}^s\oplus
\tilde{E}^u$ sufficiently close to $E_g^s\oplus E_g^u$, and hence
close to $E_f^s\oplus E_f^u$. We want to find a continuous center
section $u\in \mathfrak{X}^c $ and a homeomorphism
$\pi:M\longrightarrow M$ $\varepsilon$ close to $\id_M$ such that
(\ref{feqnThmB}) holds and
\begin{equation}\label{feqn1ThmC}
\exp_x^{-1}(\pi(x))\in \tilde{E}_x^s\oplus \tilde{E}_x^u
\end{equation}
for $x\in M$. Put $h=g\circ f^{-1}$ and $\pi(x)=\exp_x(v(x))$
for $v\in \tilde{\mathfrak{X}}^s\oplus\tilde{\mathfrak{X}}^u$, where $\tilde{\mathfrak{X}}^s$ and $\tilde{\mathfrak{X}}^u$  denote the spaces of continuous sections of $\tilde{E}^s$ and $\tilde{E}^u$ respectively. We
see that (\ref{feqnThmB}) is equivalent to
\begin{equation}\label{feqn2ThmC}
\exp_x^{-1}\circ
\exp_{h(x)}(v(h(x)))=\exp_x^{-1}\circ\tau^{(2)}_x\circ f\circ
\exp_{f^{-1}(x)}\bigl(v(f^{-1}(x))\bigr).
\end{equation}
Then we can write
\begin{equation}\label{feqn3ThmC}
\exp_x^{-1}\circ\tau^{(2)}_x\circ f\circ
\exp_{f^{-1}(x)}(v(f^{-1}(x)))
=(Fv)(x)+u(x)+\eta(w)(x),
\end{equation}
where $\omega=u+v\in \mathfrak{X}$ with
$u\in \mathfrak{X}^c$ and $v\in \tilde{\mathfrak{X}}^{us}$,
$$
(Fv)(x)=\sum_{i=s,u}\tilde{\Pi}_x^i\circ
d_{f^{-1}(x)}f\circ\tilde{\Pi}_{f^{-1}(x)}^iv(f^{-1}(x))
$$
and
\begin{eqnarray}
\eta(w)(x)&=&\exp_x^{-1}\circ\tau^{(2)}_x\circ f\circ
\exp_{f^{-1}(x)}\bigl(v(f^{-1}(x))\bigr)\notag\\
&&-d_0\bigl(\exp_x^{-1}\circ\tau^{(2)}_x\circ f\circ
\exp_{f^{-1}(x)}\bigr)v(f^{-1}(x))\notag\\
&&+\sum_{i=s,c,u,j=s,u,i\neq j}\tilde{\Pi}_x^i\circ
d_{f^{-1}(x)}f\circ\tilde{\Pi}_{f^{-1}(x)}^jv(f^{-1}(x)),
\end{eqnarray}
in which $\tilde{\Pi}_x^s$ is the projection from $T_xM$ onto
$\tilde{E}^s_x$ along $\tilde{E}^c_x\oplus \tilde{E}^u_x$,
where $\tilde{\Pi}^{c}_x={\Pi}^{c}_x$.
$\tilde{\Pi}^{c}_x$ and $\tilde{\Pi}^{u}_x$ are
defined in the similar manner. It is clear that
$F(\tilde{\mathfrak{X}}^s)=\tilde{\mathfrak{X}}^s,
 F(\tilde{\mathfrak{X}}^u)=\tilde{\mathfrak{X}}^u$.
Moreover, for any $\lambda<\tilde{\lambda}<1$, we can choose $g$ sufficiently close to $f$, $\tilde{E}^s\oplus \tilde{E}^u$
sufficiently close to $E_f^s\oplus E_f^u$ such that
$$
\|F|_{{\mathfrak{X}}^s}\|,\;\;\|F^{-1}|_{{\mathfrak{X}}^u}\|^{-1}\leq \tilde{\lambda}.
$$
Similar to \eqref{feqn3left}, we have
\begin{equation}\label{feqn4ThmC}
\exp_x^{-1}\circ\exp_{h(x)}(v(h(x)))=(J_hv+\theta_h(v))(x),
\end{equation}
where $J_h$ and $\theta_h$ is redefined with respect to $\tilde{\mathfrak{X}}=\tilde{\mathfrak{X}}^s\oplus \tilde{\mathfrak{X}}^c\oplus \tilde{\mathfrak{X}}^u$.
By \eqref{feqn3ThmC} and \eqref{feqn4ThmC},  \eqref{feqn2ThmC} is equivalent to
$$
J_hv+\theta_h(v)=Fv+u+\eta(w).
$$
From now, we can find $\pi$ and $u$ in a similar way as we have done
in the proof of Theorem B. We omit the details.

In the following, we prove that $\pi$ obtained above is a leaf conjugacy
from $(g, \mathcal{W}_g^c)$ to $(f, \mathcal{W}_f^c)$. Since the
bundle $\tilde{E}^s\oplus \tilde{E}^u$ is $C^1$, (\ref{feqn1ThmC})
implies that the restriction of $\pi$ to each center leaf of $g$ is one-to-one. If we can get that $\pi$ sends center leaves of $g$ to
that of $f$ then by the same arguments of Pesin in Lemma~5.11
of \cite{Pesin} we can conclude that $\pi$ is a leaf conjugacy.
Therefore, the remaining work is to prove that $\pi$ sends center leaves
of $g$ to that of $f$.

Now we show that for any $x\in M$, $\pi(W^c_g(x))\subset W^c_f(\pi x)$.
It is enough to show that the set $\pi(W^c_g(x))$ is tangent to $E^c_f(\pi x)$
for any $x\in M$.  Suppose not, then there exist a small number $c_1>0$
and a sequence of points $\{y_k\}\subset W^c_g(x)$ with $y_k\to x$
as $k\to \infty$ such that
\begin{equation}\label{fleafconj1}
d(\pi z_k', \pi z_k)+d(\pi z_k, \pi y_k)\ge c_1 d(\pi x,\pi y_k),
\end{equation}
where $z_k'$ and $z_k$ are the unique points such that 
$\pi z_k\in W^u_f(\pi y_k)$
and $\pi z_k'\in W^s_f(\pi z_k)\cap W^c_f(\pi x)$.
By taking a subsequence we may assume that
$d(\pi z_k, \pi y_k)\ge d(\pi z_k', \pi z_k)$ for all $k>0$,
and the other case can be discussed similarly by using $f^{-1}$.

For each $k>0$, there exists $n=n(k)>0$ such that
\begin{equation*}
d(g^i(x), g^i(y_k))\le (\mu'_g)^i d(x,y_k)\le \varepsilon
\qquad \forall \; 0\le i\le n,
\end{equation*}
where $\mu'_g$ is the upper bounds of $\|Dg|_{E^c_g}\|$ given
in the definition of partially hyperbolic diffeomorphism.
We can see that $n(k)\sim -\log d(x,y_k)/\log \mu'_g$ if $\mu'_g>1$, 
and we regard $n(k)=\infty$ if otherwise.
Since $d(\pi, \id_M)<\varepsilon$, we have
\begin{equation}\label{fleafconj2}
d\bigl(\pi(g^i x), \pi(g^i y_k)\bigr)\le 3\ve
\qquad \forall \; 0\le i\le n.
\end{equation}

Since the foliation $\mathcal{W}_f^c$ is smooth, for
any $x\in M$ there is a coordinate chart $U_x$ at $x$ of size $r>0$
such that the local leaves of the center foliation can be viewed as
parallel disks. For any $z\in U_x$, denote such a local center
disk passing through $z$ by $B_f^c(z)$.
We consider the coordinate charts $U_{\pi(g^ix)}$ at $\pi(g^i x)$
of size $r$, $i=0, 1, \cdots, n$.
We assume that $r>0$ is small and the coordinates are taken in a way
such that the metrics on the charts are close to the metric on the manifold.
Also, we assume that $\ve$ and $\delta$ are small such that
as $d(f,g)<\delta$, all the points $\pi(g^i w)$ and $f(\pi(g^{i-1} w))$
are in the chart~$U_{\pi(g^ix)}$, where $w=x, y_k, z_k, z_k'$.
Since $\tau$ is a motion along leaves, we have
$B_f^c(\pi(gx))=B_f^c(f(\pi x))$. So by the fact that
$d(f(\pi z_k), f(\pi y_k))\ge \mu_f d(\pi z_k, \pi y_k)$,
we get
$$
d(B_f^c(\pi(g z_k)), B_f^c(\pi(g y_k)))
\geq \tilde{\mu}d(B_f^c(\pi z_k), B_f^c(\pi y_k) )
$$
for some $\max\{1, \mu_g'\}< \tilde{\mu} < \mu_f$.
Inductively, we have
$$
d\bigl(B_f^c(\pi(g^n z_k)), B_f^c(\pi(g^n y_k)\bigr)
\geq \tilde{\mu}^n d\bigl(B_f^c(\pi z_k), B_f^c(\pi y_k)\bigr).
$$
Since $\pi z_k'\in W^s_f(\pi z_k)$ and
$B_f^c(\pi(g^n z_k'))=B_f^c(\pi(g^n x)$, we have
$$
d\bigl(B_f^c(\pi(g^i z_k)), B_f^c(\pi(g^i x))\bigr) \to 0 
\mbox{ as }  i\to \infty.
$$
Hence
$$
d\bigl(B_f^c(\pi(g^n x)), B_f^c(\pi(g^n y_k))\bigr)
\geq \tilde{\mu}^n d\bigl(B_f^c(\pi z_k), B_f^c(\pi y_k )\bigr).
$$
Since $\pi z_k\in W^u_f(\pi y_k)$, it is easy to see that
$d\bigl(B_f^c(\pi z_k), B_f^c(\pi y_k )\bigr)\ge c_2 d(\pi z_k, \pi y_k)$
for some constant $c_2>0$ only depends on the system.
Also by \eqref{fleafconj1} and the fact that
$d(\pi z_k, \pi y_k)\ge  d(\pi z_k', \pi z_k)$, we have
$d(\pi z_k, \pi y_k)\ge 0.5 c_1 d(\pi x, \pi y_k)$.
Since $y_k\in W_g^c(x)$ and the map $\pi(x)$ is along
$\tilde E^u_x\oplus \tilde E^s_x$, which is a smooth tangent subbundle,
we have $d(\pi x, \pi y_k) \ge c_3 d(x,y_k)$ for some $c_3>0$ independent
of $x$ and $y_k$.  Therefore, we have
$$
d\bigl(B_f^c(\pi(g^n x)), B_f^c(\pi(g^n y_k))\bigr)
\geq C \tilde{\mu}^n d(x, y_k),
$$
where $C$ is a constant independent of $x$, $y_k$ and $n$.
Since $n=n(k)$ increases like $-\log d(x,y_k)/\log \mu'_g$ and
$\min\{1,\mu'_g\}<\tilde{\mu}$, we have $\tilde{\mu}^n d(x, y_k)\to\infty$
as $d(x,y_k)\to 0$.  This contrdicts to \eqref{fleafconj2}
which implies that $d\bigl(B_f^c(\pi(g^n x)), B_f^c(\pi(g^n y_k))\bigr)$
is bounded.

Replacing the center leaves by the center-stable leaves or the
center-unstable leaves and using similar arguments in the above
paragraph, we can prove that $\pi$ sends $\mathcal{W}_g^{cs}$
and $\mathcal{W}_g^{cu}$ to $\mathcal{W}_f^{cs}$ and $\mathcal{W}_f^{cu}$ 
respectively.
\end{proof}

\section{Applications in the entropy theory}\label{SApp}

In this section, we apply our results to continuity of entropy. It
is well known that continuity properties of entropy are very
delicate. Obviously, the topological entropy of Anosov
diffeomorphisms is locally a constant since it is structurally
stable. For partially hyperbolic systems,
Hua, Saghin and Xia (\cite{Xia}) proved that for the case
that the unstable and stable foliations stably carry some
unique nontrivial homologies,
the topological entropy is locally constant if the center foliation is one
dimensional, and continuous if the center foliation is two
dimensional.
Hua (\cite{Hua}) showed that the topological entropy is continuous
at the time one map of transitive Anosov flows.
In this section, we will use our results on the structural quasi-stability
to investigate continuity of entropy for some
partially hyperbolic diffeomorphisms.

\subsection{Time one map of Anosov flow}

\begin{proof}[Proof of Theorem D]
Let $g$ be a diffeomorphism sufficiently close to $f$. By Theorem
B$'$ and Theorem~C, there exist a homeomorphism $\pi: M\rightarrow M$
with $d(\pi,\id_M)$ sufficiently small and a small
$\tilde{\tau}: M\rightarrow \mathbb{R}$ such that
$$
\pi\circ g(x)=\varphi^{\tilde{\tau}\circ f(x)}(f\circ\pi(x)) \ \ \forall
x\in M.
$$
Now we can define a flow $\psi$ by
$\psi^t(x)=\pi^{-1}\varphi^t(\pi(x))$ for $x\in M$ and
$t\in\mathbb{R}$. Obviously, $\varphi$ and $\psi$ are conjugate and
\begin{equation}\label{feqn1Thm5}
g(x)=\psi^{1+\tilde{\tau}\circ f(x)}(x)
\end {equation}
for any $x\in M$. By Theorem B of \cite{Thomas}, we have that
\begin{equation}\label{feqn2Thm5}
(1+\min_{x\in M}\tilde{\tau}(x))h(\varphi)\leq h(\psi)\leq(1+\max_{x\in
M}\tilde{\tau}(x))h(\varphi),
\end{equation}
where $h(\varphi)$ and $h(\psi)$ are the topological entropies of
$\varphi$ and $\psi$ respectively. From Proposition 21 of
\cite{Bowen0}, we have that for any $t\in\mathbb{R}$,
\begin{equation}\label{feqn3Thm5}
h(\varphi^t)=|t|h(\varphi^1)=|t|h(f) \;\;\;\text{and}\;\;\;
h(\psi^t)=|t|h(\psi^1).
\end{equation}
By (\ref{feqn1Thm5}),
\begin{equation}\label{feqn4Thm5}
h(\psi^{1+\min\limits_{{x\in M}}\tilde{\tau}(x)})\leq h(g)\leq
h(\psi^{1+\max\limits_{x\in M}\tilde{\tau}(x)}).
\end{equation}
Therefore, by (\ref{feqn2Thm5}), (\ref{feqn3Thm5}) and (\ref{feqn4Thm5}),
we have
$$
(1+\min_{x\in M}\tilde{\tau}(x))^2h(f)\leq h(g)\leq (1+\max_{x\in
M}\tilde{\tau}(x))^2h(f).
$$
Note that $|\tilde{\tau}|\rightarrow 0$ as $g\rightarrow f$. Hence we
conclude that the topological entropy function is continuous at $f$.
\end{proof}

\subsection{Systems with almost parallel center foliation}

For a smooth surface $\S$, $y\in \S$ and $r>0$, we
denote
$$
\S(y,\a)=\{z\in \S: d(z,y)< r\}.
$$

The volume growth rate of the unstable foliation of $f$ is defined by
$$
\chi^u(x, r)=\limsup_{n\to\infty}\frac{1}{n}\log
\Vol(f^n\mathcal{W}_f^u(x,r)),
$$
and
$$
\chi^u(f)=\sup_{x\in M}\chi^u(x, r).
$$
(See \cite{Xia}.) Clearly, $\chi^u(f)$ is independent of $r$.
$\chi^s(f)$ is defined similarly by using stable manifolds
$\mathcal{W}^s_f$.

\begin{proof}[Proof of Theorem E]
The first part of the theorem follows from Lemma~\ref{L1ThmE} below. By
Lemma~\ref{L2ThmE} below, the volume growth satisfies $\chi^u(f)=\chi^u(g)$.
So, following the same arguments in the proof of Theorem 1.1 in \cite{Xia},
we can obtain the second part of the theorem.
\end{proof}

Recall that $\t^c$ is a holonomy map defined by sliding the center leaves.
When we use the map, we will allow the domain to be a nonsmooth surface
or even an arbitrary set.
Also, we use $\t^c_f$ and $\t^c_g$ to denote the maps along the center leaves
of $f$ and $g$ respectively.

\begin{lemma}\label{L1ThmE}
Let $f$ be a partially hyperbolic diffeomorphism as in Theorem B
and $g$ be a diffeomorphism that is $C^1$-close to $f$. If $f$ has
almost parallel center foliation, then so does $g$.

Moreover, for any $x, y\in M$ with $y\in \mathcal{W}_f^c(\pi(x))$, and
any smooth surfaces $\S_f(x)$ and $\S_g(y)$ with
$\S_f(x)\perp \mathcal{W}^c_f$ and $\S_g(y)\perp \mathcal{W}^c_g$,
the map $\t^c_g\circ \pi: \S_f(x)\to \S_g(y)$ is uniformly
continuous with respect to $x$ and $y$.
\end{lemma}

\begin{proof}
Take $\b>0$.  Take $x\in M$ and $y\in \mathcal{W}_g^c(x)$. Let
$\S_g(x)$ and $\S_g(y)$ denote the smooth surfaces with
$\S_g(x)\perp \mathcal{W}_g^c$ and $\S_g(y)\perp \mathcal{W}_g^c$.
Denote $\S_g(x,\b)=\{z\in \S_g(x): d(x,z)\le \a\}$. We need to show
that there exists $\a>0$ independent of $x$ and $y$ such that
$\t_g^c(\S_g(x,\a))\subset \S_g(y,\b)$.

Denote $\disp {\mathcal R}_g(y,\b)=\cup_{z\in
\S_g(y,\b)}\mathcal{W}_g^c(z,\b)$, where $\mathcal{W}_g^c(z,\b)$ is
the local center leaf of $g$ at $z$ of size $\b$. Clearly,
${\mathcal R}_g(y,\b)$ contains a ball about $y$ of radius $\b$.
Since $\pi$ is a homeomorphism, $\pi^{-1}$ is uniformly
continuous on $M$. So there is $\b'>0$ independent of $y$ such that
$\pi\bigl({\mathcal R}_g(y,\b)\bigr)$ contains a ball of radius
$\b'$ about $\pi(y)$. In particular, $\pi({\mathcal
R}_g(y,\b))\supset \S_f(\pi(y),\b')$, where $\S_f(\pi(y),\b')$ is
the part of a smooth surface $\S_f(\pi(y))$ that is
contained in a ball of radius $\b'$, and $\S_f(\pi(y))\perp
\mathcal{W}_f^c$.

Note that by Theorem C,
$\pi(\mathcal{W}_g^c(y))=\mathcal{W}_f^c(\pi(y))$. Since $f$ has
almost parallel center foliation, there is $\a'>0$ independent of
$\pi(y)$ such that $\t^c_f(\S_f(\pi(x),\a'))\subset
\S_f(\pi(y),\b')$.

Consider the set $\disp {\mathcal R}_f(\pi(x),\a')=\cup_{z\in
\S_f(\pi(x),\b')}\mathcal{W}_f^c(z,\a')$. It contains a ball of
radius $\a'$ about $\pi(x)$. Since $\pi$ is uniformly continuous,
there exists $\a>0$ only depending on $\a'$, such that
$\pi^{-1}\bigl({\mathcal R}_f(\pi(x),\a')\bigr)$ contains a ball of
radius $\a$ about $x$. In particular, it contains $\S_g(x,\a)$.

Now it is easy to check that $\t^c_g\bigl(\S_g(x,\a)\bigr)\subset
\S_g(y,\b)$.

The proof of the second part of the lemma
can be obtained in a similar way.
\end{proof}

\begin{lemma}\label{L2ThmE}
Let $f$ and $g$ be as in Lemma 4.1.  Then we have
$$
\chi^u(g)=\chi^u(f).
$$
\end{lemma}

\begin{proof}
Take $x\in M$ and $r>0$.

By the last lemma we know that there exists $r^*\ge r'>0$ such that
\begin{equation}\label{fdefr}
\mathcal{W}^u_f(\pi(x), r')
\subset (\t^c \circ \pi)(\mathcal{W}^u_g(x, r))
\subset \mathcal{W}^u_f(\pi(x), r^*),
\end{equation}
where $\t^c$ is the holonomy map into
$\mathcal{W}^u_f(\pi(x), r^*)$.

Define $\psi_n=\t^c_n\circ \pi$, where $\t^c_n$ is the
holonomy map into $f^n\mathcal{W}^u_f(\pi(x), r^*)$. It is easy to
check that \eqref{fdefr} implies
\begin{equation*}
f^n \mathcal{W}^u_f(\pi(x), r')
\subset \psi_n(g^n\mathcal{W}^u_g(x, r))
 \subset f^n\mathcal{W}^u_f(\pi(x), r^*)
\end{equation*}
for any $n>0$.
Moreover, by the second part of Lemma~\ref{L1ThmE}, we know that
for any $\a>0$, there exists $\b>0$ such that
for any $x\in M$, $n>0$, and $y\in  g^n\mathcal{W}^u_g(x, r)$,
$$
\psi_n \mathcal{W}^u_g(y, \b) \subset \mathcal{W}^u_f(\psi_n(y), \a).
$$
The inclusions mean that Condition~(b) of
Sublemma~\ref{SLgrowth} below is satisfied with $\psi_n=\psi$,
$W=g^n\mathcal{W}^u_g(x, r)$ and $W'=f^n\mathcal{W}^u_f(\pi(x), r')$
for all $n\ge 0$. Since $\mathcal{W}^u_f$ and $\mathcal{W}^u_g$ are
smooth submanifolds with bounded curvature, Condition~(a) of the
sublemma holds.

Now we apply Sublemma~\ref{SLgrowth} to get that there is $C>0$
independent of $n$ such that for all $n\ge 0$,
$$
    \Vol (f^n\mathcal{W}^u_f(x, r'))
\le C \Vol (g^n\mathcal{W}^u_g(x, r)).
$$
Since this is true for any $x$, we get  $\chi^u(f)\le \chi^u(g)$.

We can apply similar arguments, by using the inverse of $\pi$ and
the fact that $g$ has almost parallel center foliation, to get
$\chi^u(g)\le \chi^u(f)$.
\end{proof}

\begin{sublemma}\label{SLgrowth}
Let $W, W', W^*\subset M$ be $k$ dimensional manifolds with
$W'\subset W^*$  and $\psi: W\to W^*$ be a one to one map
such that $W'\subset \psi(W)$. Suppose that for all $n\ge 0$,

\begin{enumerate}
\item[{\rm (a)}]
there are constants $\overline C$ and $\underline C$ such that
for any small $\a>0$, $y'\in W'$ and $y\in W$,
$$
\Vol W'(y', \a)\le \overline C\a^k, \qquad \underline C \a^k\le \Vol
W(y,\a);
$$

\item[{\rm (b)}]
for any $\a>0$, there is a constant $\b>0$ such that for any $y\in W$
with $\psi(y)\in W'$,
$$
\psi(W(y,\b)) \subset  W'(\psi(y),\a).
$$
\end{enumerate}
Then there exists $C>0$, only depending on $\underline C$,
$\overline C$, $\a$ and $\b$,  such that
$$
\Vol (W')\le C \Vol(W).
$$
\end{sublemma}

\begin{proof}
Fix $\a>0$. Take a $2\a$ separated set $y_1', \cdots, y_{\ell_n}'\in
W'$, that is, $d(y_i', y_j')\ge 2\a$ for any $1\le i, j \le \ell_n$.
We also require that the set has maximal cardinality. Hence,
$\{B_{W'}(y_i', 2\a)\}$ form a cover of $W'$. So by part (a) we have
$$
\Vol(W')\le \ell_n\cdot {\overline C}(2\a)^{k}.
$$

Take $\b>0$ as in Condition (b) of the sublemma.
Since the balls in $\{W'(y_i', \a)\}$
are pairwise disjoint, and $\psi W(y_i, \b)\subset B_{W'}(y_i',
\a)$, where $y_i=\psi^{-1}y_i'\in W$, we see that $\{W(y_i, \b)\}$
are pairwise disjoint. Hence,
$$
\Vol(W)\ge \ell_n\cdot {\underline C}\b^{k}.
$$

So we have
$$
\Vol(W')\le C \Vol(W),
$$
where $\disp C= {\overline C}(2\a)^{k}/{\underline C}\b^{k}$.
\end{proof}


\begin{thebibliography}{99}

\bibitem{Anosov1} D. Anosov, Geodesic flows on closed Riemann manifolds with negative curvature, \emph{Proc. Steklov Inst. Math.} {\bf 90} 1967.

\bibitem {Barreira} L. Barreira and Ya. Pesin, Nonuniform hyperbolicity.
Dynamics of systems with nonzero Lyapunov exponents. Encyclopedia of Mathematics and its Applications, 115. \emph{Cambridge University Press}, Cambridge, 2007.

\bibitem{Bonatti} C. Bonatti, L. Diaz and M. Viana,
Dynamics beyond uniform hyperbolicity: A global geometric and probabilistic
perspective, \emph{Encyclopaedia Math. Sci.}, 102, Springer-Verlag, Berlin,
2005.

\bibitem{Bowen0} R. Bowen,
Entropy for group endomorphisms and homogenuous spaces,
                \emph{Trans. Amer. Math. Soc.} {\bf 153} (1971), 401-414.

\bibitem {Brin} M. Brin and Ya. Pesin,
Partially hyperbolic dynamical systems,
\emph{MAth. USSR-Izv.}, {\bf 8} (1974), 177-218.

\bibitem {Hirsch} M. Hirsch, C. Pugh and M. Shub,
Invariant manifolds, \emph{Bull. Amer. Math. Soc.}, {\bf 76} (1970), 1015-1019.

\bibitem {Hirsch1} M. Hirsch, C. Pugh and M. Shub, Invariant
Manifolds, Vol 583 of Lect. Notes in Math., Springer Verlag, 1977.

\bibitem {HZZ} H. Hu, Y. Zhou and Y. Zhu,
Quasi-shadowing for partially hyperbolic diffeomorphisms, preprint.

\bibitem {Hua}  Y. Hua, Continuity of topological entropy at time one map of transitive Anosov flows, Thesis (Ph.D.) Northwestern University. 2009.

\bibitem {Xia} Y. Hua, R. Saghin and Z. Xia, Topological entropy and partially hyperbolic diffeomorphisms, \emph{Ergod. Th. Dynam. Sys.}, {\bf 28} (2008), 843-862.

\bibitem {Kato} K. Kato and A. Morimoto,
Topological stability of Anosov flows and their centerizers,
\emph{Topology}, {\bf 12} (1973), 255--273.

\bibitem {Pesin} Ya. Pesin, Lectures on partial hyperbolicity and
stable rrgodicity, Zurich Lectures in Advanced Mathematics. European
Mathematical Society (EMS), Zurich, 2004.

\bibitem {Pugh}C. Pugh and M. Shub, Stably ergodic dynamical systems
and partial hyperbolicity, \emph{J. Complexity} {\bf 13} (1997), 125-179.

\bibitem {Thomas} R.F. Thomas, Entropy of expansive flows,
                  \emph{Ergo. Theo. Dyn. Syst.} {\bf 7} (1987), 611-625.

\bibitem {Walters} P. Walters, Anosov diffeomorphisms are
topologically stable, \emph{Topology}, {\bf 9} (1970), 71-78.

\bibitem {Walters1}  P. Walters, On the pseudo-orbit tracing property and its relationship to stability, Vol 668 of Lect. Notes in Math., Springer Verlag, 1978, 231-244.

\end{thebibliography}
\end{document}